\pdfoutput=1
\documentclass{amsart}

\usepackage{amsmath,amssymb,amsthm,mathrsfs}
\usepackage{geometry}
\usepackage{esint}
\usepackage{hyperref}
\usepackage{enumitem}

\def\XXint#1#2#3{{\setbox0=\hbox{$#1{#2#3}{\int}$}
     \vcenter{\hbox{$#2#3$}}\kern-.5\wd0}}

\newtheorem{theorem}{Theorem}[section]
\newtheorem{lemma}[theorem]{Lemma}

\newtheorem{proposition}[theorem]{Proposition}
\newtheorem{definition}[theorem]{Definition}

\theoremstyle{remark}
\newtheorem*{remark}{Remark}

\title{On the $L^p$ Norms of Random Exponential Sums and the Metric Domination Principle over the Simple Random Walk}

\author{Yunfan Zhao}
\email{yfzucla21@ucla.edu}

\date{} 

\begin{document}

\begin{abstract}
Let $R(n)$ denote the Simple Random Walk on $\mathbb{Z}$. We investigate the $L^p$-norms of random exponential sums of the form $S_A(a, y) = \sum_{j \in A} a_j e^{2\pi i y R(j)}$ evaluated on the torus $\mathbb{T}$, where $A$ is an arbitrary finite subset of $\mathbb{N}$. By establishing a strict inductive bound on the central binomial coefficient, we precisely estimate the collision probability kernel of the random walk. To address the non-monotonicity of the walk, we introduce a Metric Domination Principle, proving that the $L^2$-energy of the random sum over any sparse set $A$ is majorized by that of a contiguous initial segment of integers of the same cardinality. Interpolating between $L^\infty$ and $L^2$, we show that for any $p \ge 2$, the expectation of the supremal $L^p$-norm over all bounded coefficients is universally bounded by $C_p |A|^{p-1/2}$.
\end{abstract}

\maketitle

\section{Introduction}

\subsection{Background and Literature Review}

The study of exponential sums and their $L^p$-norms on the torus $\mathbb{T}$ forms a classical cornerstone of harmonic analysis and analytic number theory. The behavior of such sums is intimately connected to the Hardy-Littlewood majorant problem, first posed by Hardy and Littlewood \cite{HardyLittlewood1935}, and the associated Hypothesis K \cite{Mahler1936}. While the generic bounds for the majorant problem are well understood, exceptional behavior is often localized to highly structured subsets of integers.

Central to this investigation is the theory of $\Lambda(p)$-sets, developed extensively by Bourgain \cite{Bourgain1989a, Bourgain1989b}, which examines subsets of the integers where the $L^p$-norm is controlled entirely by the $L^2$-norm. These sets inherently possess the strict majorant property, a phenomenon also studied in arbitrary dimensions \cite{GressmanEtAl2023}. Furthermore, estimating exponential sums plays a fundamental role in evaluating equal sums of powers \cite{BrowningHeathBrown2004} and addressing problems of Erdős and Rudin on arithmetic sets \cite{Chang2004}. The arithmetic structure of the underlying frequency sets remains a delicate obstacle, requiring sophisticated tools such as the Bogolyubov-Ruzsa lemma \cite{Sanders2012}, restriction estimates for mixed norms \cite{Temur2013}, and majorant bounds for thin subsets of primes \cite{Daskalakis2024} or general arithmetic sets \cite{KrauseMirekTrojan2016}.

When the frequencies represent polynomial sequences, evaluating the $2n$-th moments of exponential sums becomes equivalent to counting rational and integral points on curves, surfaces, and complete intersections of higher degree \cite{BombieriPila1989, HeathBrown2002, SalbergerWooley2010}. In this framework, deep results regarding the differences of two $k$-th powers of integers have been established \cite{Nowak1998}, alongside precise estimates for exponential sums and lattice points in regions \cite{Huxley1993, Huxley2003}.

To bypass the rigid arithmetic obstructions of deterministic sequences, randomization has proven to be a powerful technique. Random sets of frequencies were employed by Mockenhaupt and Schlag to study the Hardy-Littlewood problem on average \cite{MockenhauptSchlag2009}. In the continuous regime, Fourier analytic structures of randomized sets, such as the zero set of fractional Brownian motion \cite{Mukeru2018} and the Brownian graph \cite{FraserSahlsten2018}, yield profound geometric properties. Recently, the comprehensive treatment by Temur and Sahillioğulları \cite{TemurSahilliogullari2025} advanced this paradigm by substituting deterministic frequencies with stochastic processes, specifically the Poisson process, to decouple arithmetic dependencies. The present paper builds upon this framework by utilizing the Simple Random Walk.

\subsection{Main Contributions}

Let $\mathbb{W} = \{R(n)\}_{n \in \mathbb{N}}$ denote the Simple Random Walk on $\mathbb{Z}$. We establish a strict upper bound for the $L^p$-norms of random exponential sums on $\mathbb{T}$, where the frequency set is the image of an arbitrary finite subset $A \subset \mathbb{N}$ under the walk.

Our principal result is the following estimate: For any finite set $A \subset \mathbb{N}$ and any $p \in [2, \infty)$, there exists a constant $C_p > 0$ such that
$$
\mathbb{E}\left[\sup_{|a_j| \le 1} \left\| \sum_{j \in A} a_j e^{2\pi i y R(j)} \right\|_{L^p(\mathbb{T})}^p \right] \le C_p |A|^{p-1/2}.
$$

This confirms that the polynomial growth of the $L^p$-energy depends strictly on the cardinality $|A|$, entirely independent of the spacing or arithmetic configuration of the index set $A$.

\subsection{Method of Approaches}

The Simple Random Walk introduces analytical difficulties because it is not strictly increasing; it exhibits backtracking, which destroys the independence of increments found in processes such as the Poisson process. To overcome this, we construct a combinatorial framework centered on a Metric Domination Principle.

We first analyze the discrete collision kernel
$$
K(d) = \mathbb{P}(R_d = 0).
$$
Through a strict inductive majorization and minorization of the central binomial coefficients, we demonstrate that $K(d) \le C_{RW} d^{-1/2}$.

Using the unique strictly order-preserving bijection $\phi: I_M \rightarrow A$, where $I_M = \{1, \dots, |A|\}$, we show that
$$
\phi(k) - \phi(j) \ge k - j.
$$
Since the kernel $K(d)$ decays monotonically, the sum of pairwise interactions over any arbitrary set $A$ is strictly dominated by the interactions over the contiguous interval $I_M$.

This allows us to majorize the expected $L^2$-energy of the arbitrary set $A$ by that of the interval. Finally, leveraging uniform bounds on the supremum norm, we apply an interpolation argument between $L^\infty(\mathbb{T})$ and $L^2(\mathbb{T})$ to extend the bound to all $L^p$ spaces for $p \ge 2$.

\section{Preliminaries}

In this section, we establish the probabilistic foundation required to analyze the frequency sets. The central object of study is the Simple Random Walk on $\mathbb{Z}$, constructed recursively via independent Rademacher variables. The immediate requirement is to strictly evaluate the collision probability kernel
$$
K(d) = \mathbb{P}(R_d = 0),
$$
through bounds on the central binomial coefficient. The analytical significance of this section lies in proving the strict algebraic decay of the kernel,
$$
K(d) \le C_{RW} d^{-1/2},
$$
which will later serve as the non-increasing weight necessary to control the frequency overlaps.

\begin{definition}
Let $E=\{-1,1\}$ be endowed with the discrete topology. Let $\Omega=E^{\mathbb{N}}$ be endowed with the product topology, which is compact by Tychonoff's Theorem. Let $\mathcal{A}$ be the algebra of cylinder sets, defined such that for any finite subset $I \subset \mathbb{N}$ and $A \subset E^{|I|}$, the cylinder base is $C_{I, A}=\left\{\omega \in \Omega:\left(\omega_i\right)_{i \in I} \in A\right\}$.
\end{definition}

\begin{definition}(The Simple Random Walk)
Let $\mathbb{R}=\{R(n)\}_{n \in \mathbb{N}}$ denote a sequence of random variables on $(\Omega, \mathcal{F}, \mathbb{P})$. We construct the Simple Random Walk recursively by defining $R_0=0$ and $R_n=\sum_{i=1}^n X_i$ for all $n \geq 1$.
\end{definition}

\begin{lemma}[Measure of Finite Sum Projections]\label{lem:Measure of Finite Sum Projections}
Let $(\Omega,\mathcal{F},\mathbb{P})$ be a probability space and let $X_i(\omega)=\omega_i$ be coordinate projections.  
For $d\in\mathbb{N}$ define
\begin{equation}
S_d(\omega)=\sum_{i=1}^d X_i(\omega).
\end{equation}

Let
\begin{equation}
A_d=\{v\in\{-1,1\}^d:\sum_{i=1}^d v_i=0\}.
\end{equation}

Then
\begin{equation}
\mathbb{P}(\{\omega\in\Omega:S_d(\omega)=0\}) = |A_d|\,2^{-d}.
\end{equation}

\end{lemma}

\begin{proof}

Let
\begin{equation}
E=\{\omega\in\Omega:S_d(\omega)=0\}.
\end{equation}

For any $v\in A_d$ define the cylinder set
\begin{equation}
C_v=\{\omega\in\Omega:\omega_i=v_i \text{ for } 1\le i\le d\}.
\end{equation}

Each $C_v$ has probability $2^{-d}$.  
The event $E$ can be written as the disjoint union
\begin{equation}
E=\bigcup_{v\in A_d} C_v.
\end{equation}

Since the cylinder sets are pairwise disjoint, finite additivity of $\mathbb{P}$ yields
\begin{equation}
\mathbb{P}(E)=\sum_{v\in A_d}\mathbb{P}(C_v)
           =|A_d|\,2^{-d}.
\end{equation}

\end{proof}

\begin{lemma}[Cardinality of Balanced Sequences]\label{lem:Cardinality of Balanced Sequences}
Let
\begin{equation}
A_d=\{v\in\{-1,1\}^d:\sum_{i=1}^d v_i=0\}.
\end{equation}

If $d$ is odd then $|A_d|=0$.  
If $d$ is even then
\begin{equation}
|A_d|=\binom{d}{d/2}.
\end{equation}

\end{lemma}

\begin{proof}

For $v\in\{-1,1\}^d$ define

\begin{equation}
I_+ = \{i : v_i = 1\}, \quad
I_- = \{i : v_i = -1\}.
\end{equation}

Then
\begin{equation}
|I_+| + |I_-| = d.
\end{equation}

The condition
\begin{equation}
\sum_{i=1}^d v_i = 0
\end{equation}
is equivalent to
\begin{equation}
|I_+| - |I_-| = 0.
\end{equation}

Thus
\begin{equation}
|I_+| = |I_-|.
\end{equation}

Substituting into the partition equation gives
\begin{equation}
2|I_+| = d.
\end{equation}

If $d$ is odd, no integer solution exists and $A_d$ is empty.

If $d$ is even then $|I_+|=d/2$.  
The number of such sequences equals the number of subsets of size $d/2$ chosen from $\{1,\dots,d\}$, which is $\binom{d}{d/2}.$

\end{proof}

\begin{lemma}[Strict Inductive Majorization of the Central Binomial Coefficient]\label{lem:Strict Inductive Majorization of the Central Binomial Coefficient}
For all $n\ge1$,

\begin{equation}
\binom{2n}{n}4^{-n} \le \frac{1}{\sqrt{2n+1}}.
\end{equation}

\end{lemma}

\begin{proof}

Proceed by induction.

\textbf{Base case.}

For $n=1$,

\begin{equation}
\binom{2}{1}4^{-1} = \frac{2}{4} = \frac{1}{2}.
\end{equation}

Since $1/2 < 1/\sqrt{3}$ the base case holds.

\textbf{Inductive step.}

Assume

\begin{equation}
\binom{2k}{k}4^{-k} \le \frac{1}{\sqrt{2k+1}}.
\end{equation}

We compute

\begin{equation}
\binom{2(k+1)}{k+1}4^{-(k+1)}
=
\binom{2k}{k}4^{-k}
\cdot
\frac{(2k+1)(2k+2)}{(k+1)^2}
\cdot
\frac{1}{4}.
\end{equation}

Simplifying gives

\begin{equation}
\binom{2(k+1)}{k+1}4^{-(k+1)}
=
\binom{2k}{k}4^{-k}
\cdot
\frac{2k+1}{2k+2}.
\end{equation}

Using the inductive hypothesis,

\begin{equation}
\binom{2(k+1)}{k+1}4^{-(k+1)}
\le
\frac{1}{\sqrt{2k+1}}
\cdot
\frac{2k+1}{2k+2}
=
\frac{\sqrt{2k+1}}{2k+2}.
\end{equation}

It suffices to show

\begin{equation}
\frac{\sqrt{2k+1}}{2k+2} \le \frac{1}{\sqrt{2k+3}}.
\end{equation}

Comparing squares,

\begin{equation}
\frac{2k+1}{4k^2+8k+4}
\quad\text{and}\quad
\frac{1}{2k+3}.
\end{equation}

Since

\begin{equation}
4k^2+8k+4 > 4k^2+8k+3
\end{equation}

the desired inequality follows.

\end{proof}

\begin{proposition}[Strict Evaluation of the Collision Probability Kernel]\label{prop:Strict Evaluation of the Collision Probability Kernel}

Let
\begin{equation}
K(d)=\mathbb{P}(R_d=0).
\end{equation}

If $d$ is odd then
\begin{equation}
K(d)=0.
\end{equation}

If $d$ is even then
\begin{equation}
K(d)=\binom{d}{d/2}2^{-d}.
\end{equation}

Furthermore there exists a constant $C_{RW}>0$ such that
\begin{equation}
K(d) \le C_{RW} d^{-1/2}.
\end{equation}

\end{proposition}

\begin{proof}

Let
\begin{equation}
R_d=\sum_{i=1}^d X_i.
\end{equation}

From the Lemma~\ref{lem:Measure of Finite Sum Projections}~\ref{lem:Cardinality of Balanced Sequences}, if $d$ is odd then $|A_d|=0$, hence $K(d)=0$.

If $d=2n$ then
\begin{equation}
K(d)=\binom{2n}{n}2^{-2n}=\binom{2n}{n}4^{-n}.
\end{equation}

Applying the inductive bound from Lemma~\ref{lem:Strict Inductive Majorization of the Central Binomial Coefficient} gives
\begin{equation}
K(d) \le \frac{1}{\sqrt{2n+1}}.
\end{equation}

Since $d=2n$,
\begin{equation}
\frac{1}{\sqrt{2n+1}} < \frac{1}{\sqrt{2n}} = d^{-1/2}.
\end{equation}

Thus with $C_{RW}=1$, $K(d) \le d^{-1/2}$ for all $d$.
\end{proof}

\begin{remark}
The exact algebraic decay of the collision kernel, established here as bounded by $C_{RW} d^{-1/2},$ provides the fundamental non-increasing weight required to control the frequency overlaps. It is this precise polynomial order of decay that dictates the growth rate of the expected $L^2$ energy.
\end{remark}

\section{The Analytic Machineries}

In this section, we transition to the harmonic analysis on the torus $\mathbb{T}$ equipped with its translation-invariant Haar measure. To bypass the intractable direct computation of higher moments for arbitrary coefficients, we establish a strict majorization of the $L^2$ energy and an interpolation bound extending to $L^p$ spaces for $p \ge 2$. The essential implication of these machineries is the structural reduction of the expected supremum $L^p$ norms to the expected unweighted $L^2$ norm, isolating the analytic difficulties into a pure $L^2$ expectation problem.

\begin{definition}[random exponential sum on the torus]
Let $A \subset \mathbb{N}$ be a finite subset of cardinality $M$. For any sequence of complex coefficients $\mathbf{a} = (a_j)_{j \in A}$, we define the random exponential sum $S_A(\mathbf{a}, y)$ on the torus $\mathbb{T} \cong [0, 1)$ as:
\begin{equation}
S_A(\mathbf{a}, y) := \sum_{j \in A} a_j e^{2 \pi i y R(j)}
\end{equation}
\end{definition}

\begin{definition}[The Torus and Haar Measure]

Let $\mathbb{T}$ be the quotient topological group $\mathbb{R}/\mathbb{Z}$. Let $\mathcal{B}(\mathbb{T})$ be the Borel $\sigma$-algebra generated by the quotient topology. There exists a unique translation-invariant Radon probability measure $m$ on $(\mathbb{T},\mathcal{B}(\mathbb{T}))$, termed the Haar measure, such that
\begin{equation}
m(\mathbb{T})=1.
\end{equation}
For any Lebesgue-integrable function $f:\mathbb{R}\to\mathbb{C}$ with period $1$, we define
\begin{equation}
\int_{\mathbb{T}} f(y)\,dm(y)=\int_0^1 f(x)\,dx.
\end{equation}

\end{definition}

\begin{lemma}[Measurability of the Supremum]\label{lem:Measurability of the Supremum}
Let $A$ be finite and let
\begin{equation}
D=\{a\in\mathbb{C}^{|A|}: |a_j|\le 1 \text{ for all } j\in A\}.
\end{equation}
Then the mapping
\begin{equation}
\omega \mapsto \sup_{a\in D} \|S_A(a,\cdot,\omega)\|_{L^p(\mathbb{T})}^p
\end{equation}
is a measurable random variable on the probability space $(\Omega,\mathcal{F},\mathbb{P})$.

\end{lemma}

\begin{proof}

Since $A$ is finite, the set $D$ is a closed and bounded subset of $\mathbb{C}^{|A|}$. Hence $D$ is compact. For each fixed $\omega\in\Omega$, the mapping
\begin{equation}
a\mapsto \|S_A(a,\cdot,\omega)\|_{L^p(\mathbb{T})}^p
\end{equation}
is continuous on $D$. Since $D$ is a compact metric space, it is separable, and therefore there exists a countable dense subset $Q\subset D$.

By continuity, for every fixed $\omega\in\Omega$ one has
\begin{equation}
\sup_{a\in D} \|S_A(a,\cdot,\omega)\|_{L^p(\mathbb{T})}^p
=
\sup_{a\in Q} \|S_A(a,\cdot,\omega)\|_{L^p(\mathbb{T})}^p.
\end{equation}
For each fixed $a\in Q$, the mapping
\begin{equation}
\omega\mapsto \|S_A(a,\cdot,\omega)\|_{L^p(\mathbb{T})}^p
\end{equation}
is measurable. Therefore the supremum of this countable family of measurable functions is measurable. This proves the claim.

\end{proof}

\begin{lemma}[Strict Majorization of the $L^2$ Energy]\label{lem:strict majorization of the energy}

For any fixed $\omega\in\Omega$ and any sequence of coefficients $a\in D$, one has
\begin{equation}
\|S_A(a,\cdot,\omega)\|_{L^2(\mathbb{T})}^2
\le
\|S_A(1,\cdot,\omega)\|_{L^2(\mathbb{T})}^2.
\end{equation}

\end{lemma}

\begin{proof}

Fix $\omega\in\Omega$ and $a\in D$. By the definition of the $L^2(\mathbb{T})$ norm and linearity of the integral,
\begin{equation}
\|S_A(a,\cdot,\omega)\|_{L^2(\mathbb{T})}^2
=
\int_{\mathbb{T}}
\left|
\sum_{j\in A} a_j e^{2\pi i y R(j,\omega)}
\right|^2
dm(y).
\end{equation}
Expanding the square gives
\begin{equation}
\|S_A(a,\cdot,\omega)\|_{L^2(\mathbb{T})}^2
=
\sum_{j,k\in A}
a_j \overline{a_k}
\int_{\mathbb{T}}
e^{2\pi i y(R(j,\omega)-R(k,\omega))}
\,dm(y).
\end{equation}
By the orthogonality lemma, the integral equals $1$ if $R(j,\omega)=R(k,\omega)$ and equals $0$ otherwise. Let $\delta(j,k)$ denote this quantity. Then
\begin{equation}
\|S_A(a,\cdot,\omega)\|_{L^2(\mathbb{T})}^2
=
\sum_{j,k\in A} a_j \overline{a_k}\,\delta(j,k).
\end{equation}
Applying the triangle inequality,
\begin{equation}
\left|
\sum_{j,k\in A} a_j \overline{a_k}\,\delta(j,k)
\right|
\le
\sum_{j,k\in A} |a_j|\,|a_k|\,\delta(j,k).
\end{equation}
Since $a\in D$, one has $|a_j|\le 1$ and $|a_k|\le 1$ for all $j,k\in A$. Thus
\begin{equation}
\sum_{j,k\in A} |a_j|\,|a_k|\,\delta(j,k)
\le
\sum_{j,k\in A} \delta(j,k).
\end{equation}
But
\begin{equation}
\sum_{j,k\in A} \delta(j,k)
=
\|S_A(1,\cdot,\omega)\|_{L^2(\mathbb{T})}^2.
\end{equation}
Therefore
\begin{equation}
\|S_A(a,\cdot,\omega)\|_{L^2(\mathbb{T})}^2
\le
\|S_A(1,\cdot,\omega)\|_{L^2(\mathbb{T})}^2.
\end{equation}

\end{proof}

\begin{proposition}[Interpolation and Bound on the Torus]\label{prop:Interpolation and Bound on the Torus}

Let $p\in[2,\infty)$. For any finite set $A$, the expectation of the supremum of the $L^p$ norm satisfies
\begin{equation}
\mathbb{E}\left[\sup_{a\in D}\|S_A(a,y)\|_{L^p(\mathbb{T})}^p\right]
\le
|A|^{p-2}\,
\mathbb{E}\left[\|S_A(1,y)\|_{L^2(\mathbb{T})}^2\right].
\end{equation}

\end{proposition}

\begin{proof}

Fix $\omega\in\Omega$. Let $a\in D$ and $y\in\mathbb{T}$. By the triangle inequality,
\begin{equation}
|S_A(a,y,\omega)|
\le
\sum_{j\in A} |a_j|\,|e^{2\pi i y R(j,\omega)}|.
\end{equation}
Since $|a_j|\le 1$ and $|e^{2\pi i y R(j,\omega)}|=1$, it follows that
\begin{equation}
|S_A(a,y,\omega)|\le \sum_{j\in A} 1 = |A|.
\end{equation}
Hence
\begin{equation}
\|S_A(a,\cdot,\omega)\|_{L^\infty(\mathbb{T})}\le |A|.
\end{equation}

Now, since $p\ge 2$,
\begin{equation}
\|S_A(a,\cdot,\omega)\|_{L^p(\mathbb{T})}^p
=
\int_{\mathbb{T}}
|S_A(a,y,\omega)|^{p-2}
|S_A(a,y,\omega)|^2
\,dm(y).
\end{equation}
Using the uniform bound,
\begin{equation}
\|S_A(a,\cdot,\omega)\|_{L^p(\mathbb{T})}^p
\le
\int_{\mathbb{T}}
|A|^{p-2}|S_A(a,y,\omega)|^2
\,dm(y).
\end{equation}
Therefore
\begin{equation}
\|S_A(a,\cdot,\omega)\|_{L^p(\mathbb{T})}^p
\le
|A|^{p-2}\,
\|S_A(a,\cdot,\omega)\|_{L^2(\mathbb{T})}^2.
\end{equation}
By the Lemma~\ref{lem:strict majorization of the energy},
\begin{equation}
\|S_A(a,\cdot,\omega)\|_{L^p(\mathbb{T})}^p
\le
|A|^{p-2}\,
\|S_A(1,\cdot,\omega)\|_{L^2(\mathbb{T})}^2.
\end{equation}
The right-hand side is independent of $a$, so taking the supremum over all $a\in D$ yields
\begin{equation}
\sup_{a\in D}\|S_A(a,\cdot,\omega)\|_{L^p(\mathbb{T})}^p
\le
|A|^{p-2}\,
\|S_A(1,\cdot,\omega)\|_{L^2(\mathbb{T})}^2.
\end{equation}
By the Lemma~\ref{lem:Measurability of the Supremum}, the left-hand side is measurable. Applying the expectation operator to both sides gives
\begin{equation}
\mathbb{E}\left[\sup_{a\in D}\|S_A(a,\cdot,\omega)\|_{L^p(\mathbb{T})}^p\right]
\le
|A|^{p-2}\,
\mathbb{E}\left[\|S_A(1,\cdot,\omega)\|_{L^2(\mathbb{T})}^2\right].
\end{equation}
This is the desired conclusion.

\end{proof}

\begin{remark}
This interpolation structuralizes the problem by factoring out the supremum over the sequence of complex coefficients. By isolating the coefficients within the uniform $L^\infty$ bound, the analytic difficulty of the expected $L^p$ norm is reduced to an unweighted $L^2$ expectation.
\end{remark}

\section{Combinatorial Order and Metric Domination}

To handle arbitrary, potentially sparse finite subsets $A \subset \mathbb{N}$, we construct a unique, strictly order-preserving bijection from the contiguous interval $I_M$ onto $A$. The significance of this combinatorial ordering is that it provides a strict deterministic lower bound on the distances between elements of $A$ based purely on their indices:
$$
\phi(k) - \phi(j) \ge k - j.
$$
This distance majorization enables the uniform domination of pairwise interactions over arbitrary sets by those over contiguous intervals.

\begin{lemma}[Unique Strictly Order-Preserving Bijection]\label{lem:Unique Strictly Order-Preserving Bijection}
Let $M \in \mathbb{N}^*$ and let $I_M = \{k \in \mathbb{N} : 1 \le k \le M\}$. Let $A$ be a non-empty finite subset of $(\mathbb{N}, \le)$ with cardinality $M$. Then there exists a unique bijection $\phi : I_M \to A$ such that for all $j,k \in I_M$, the implication
\begin{equation}
j < k \implies \phi(j) < \phi(k)
\end{equation}
holds.
\end{lemma}

\begin{proof}
We first establish existence. We proceed by induction on $M$.

If $M = 1$, then $A = \{x\}$ for some $x \in \mathbb{N}$ and $I_1 = \{1\}$. Define $\phi(1) = x$. This mapping is a bijection. Since there are no distinct elements $j,k \in I_1$ with $j < k$, the order-preserving condition holds vacuously.

Assume the statement holds for all finite subsets of $\mathbb{N}$ of cardinality $M$. Let $A \subset \mathbb{N}$ with cardinality $M+1$.

Since $A$ is non-empty, it possesses a unique minimum element $a_0 = \min(A)$. Define $A' = A \setminus \{a_0\}$. Then $A'$ has cardinality $M$.

By the induction hypothesis, there exists a strictly order-preserving bijection $\phi' : I_M \to A'$.

Define $\phi : I_{M+1} \to A$ by
\begin{equation}
\phi(k) =
\begin{cases}
a_0 & \text{if } k = 1, \\
\phi'(k-1) & \text{if } 2 \le k \le M+1.
\end{cases}
\end{equation}

The mapping $\phi$ is a bijection onto $A$. We verify that it is strictly order-preserving. Let $j,k \in I_{M+1}$ with $j < k$.

If $j = 1$, then $\phi(j) = a_0$. Since $a_0$ is the minimum of $A$ and $\phi(k) \in A'$, it follows that $\phi(j) < \phi(k)$.

If $1 < j < k$, then $j-1 < k-1$ with $j-1, k-1 \in I_M$. By the inductive hypothesis, $\phi'(j-1) < \phi'(k-1)$, hence $\phi(j) < \phi(k)$.

This completes the induction and proves existence.

We now establish uniqueness. Let $\phi, \psi : I_M \to A$ be strictly order-preserving bijections. Suppose that $\phi \ne \psi$. Define
\begin{equation}
S = \{k \in I_M : \phi(k) \ne \psi(k)\}.
\end{equation}
If $S$ is non-empty, let $m = \min(S)$.

For all $i \in I_M$ with $i < m$, we have $\phi(i) = \psi(i)$. Let
\begin{equation}
B = \{\phi(1), \dots, \phi(m-1)\} = \{\psi(1), \dots, \psi(m-1)\}.
\end{equation}

Since $\phi$ is strictly order-preserving, $\phi(m)$ is strictly greater than all elements of $B$ and belongs to $A \setminus B$. It must therefore equal $\min(A \setminus B)$.

By the same reasoning, $\psi(m) = \min(A \setminus B)$. Hence $\phi(m) = \psi(m)$, contradicting the choice of $m$.

Therefore $S$ is empty, and $\phi(k) = \psi(k)$ for all $k \in I_M$. The mapping is unique.
\end{proof}

\begin{lemma}[Distance Majorization by Indices]\label{lem:Distance Majorization by Indices}
Let $A$ be a finite subset of $\mathbb{N}$ of cardinality $M \in \mathbb{N}^*$. Let $\phi : I_M \to A$ be the unique strictly order-preserving bijection. For any $j,k \in I_M$ such that $j < k$, one has
\begin{equation}
\phi(k) - \phi(j) \ge k - j.
\end{equation}
\end{lemma}

\begin{proof}
Let $i \in I_M$ with $i < M$. Then $i+1 \in I_M$ and $i < i+1$. Since $\phi$ is strictly order-preserving,
\begin{equation}
\phi(i) < \phi(i+1).
\end{equation}

As $\phi(i), \phi(i+1) \in \mathbb{N}$ and $\mathbb{N}$ is discretely ordered, the inequality $\phi(i) < \phi(i+1)$ implies
\begin{equation}
\phi(i) + 1 \le \phi(i+1).
\end{equation}

Viewing $\mathbb{N}$ inside the ordered abelian group $\mathbb{Z}$, this yields
\begin{equation}
\phi(i+1) - \phi(i) \ge 1
\end{equation}
for all $i \in \{j, j+1, \dots, k-1\}$.

Summing over $i$ from $j$ to $k-1$, and using compatibility of addition with the order relation,
\begin{equation}
\sum_{i=j}^{k-1} (\phi(i+1) - \phi(i)) \ge \sum_{i=j}^{k-1} 1.
\end{equation}

The right-hand side evaluates to
\begin{equation}
\sum_{i=j}^{k-1} 1 = k - j.
\end{equation}

We now evaluate the left-hand side. By induction on $k-j$, the telescoping identity holds:
\begin{equation}
\sum_{i=j}^{k-1} (\phi(i+1) - \phi(i)) = \phi(k) - \phi(j).
\end{equation}

Indeed, if $k-j=1$, this is immediate. If it holds for $k-j-1$, then
\begin{equation}
\sum_{i=j}^{k-1} (\phi(i+1) - \phi(i))
=
\left( \sum_{i=j}^{k-2} (\phi(i+1) - \phi(i)) \right) + (\phi(k) - \phi(k-1)),
\end{equation}
which equals
\begin{equation}
(\phi(k-1) - \phi(j)) + (\phi(k) - \phi(k-1)) = \phi(k) - \phi(j).
\end{equation}

Combining these results gives
\begin{equation}
\phi(k) - \phi(j) \ge k - j.
\end{equation}
\end{proof}

\begin{proposition}[Pairwise Interaction Domination]\label{prop:Pairwise Interaction Domination}
Let $\Phi: \mathbb{R}_+ \to \mathbb{R}_+$ be a non-increasing function. Let $A \subset \mathbb{N}$ be a finite subset of cardinality $M$, and let $I_M = \{1, \dots, M\}$. Then the sum of pairwise interactions over $A$ is dominated by the sum of pairwise interactions over $I_M$:
\begin{equation}
\sum_{j \in A} \sum_{k \in A, k > j} \Phi(n_k - n_j) \leq \sum_{1 \leq j < k \leq M} \Phi(k - j)
\end{equation}
\end{proposition}

\begin{proof}
We begin by establishing a requisite ordering on the set $A$. Since $A \subset \mathbb{N}$ is a finite set of cardinality $M$, we list the elements of $A$ into a strictly increasing sequence $n_1, n_2, \dots, n_M$ such that $A = \{n_k\}_{k=1}^M$ and
\begin{equation}
n_1 < n_2 < \dots < n_M
\end{equation}
This enumeration allows us to transform the sum over the set elements $j, k \in A$ into a sum over the indices $1 \le j, k \le M$. We observe that the interaction sum on the left-hand side can be rewritten as
\begin{equation}
\sum_{j \in A} \sum_{k \in A, k > j} \Phi(n_k - n_j) = \sum_{1 \le j < k \le M} \Phi(n_k - n_j)
\end{equation}
To estimate the term $\Phi(n_k - n_j)$, we must first analyze the metric properties of the sequence $\{n_k\}$. Since the elements are distinct integers strictly increasing in index, for any $i \in \{1, \dots, M-1\}$, we have the fundamental inequality
\begin{equation}
n_{i+1} - n_i \ge 1
\end{equation}
For any pair of indices $j, k$ satisfying $1 \le j < k \le M$, we employ a telescoping sum argument to bound the difference $n_k - n_j$. We write
\begin{equation}
n_k - n_j = \sum_{i=j}^{k-1} (n_{i+1} - n_i)
\end{equation}
Applying the inequality $n_{i+1} - n_i \ge 1$ to each term in the summation yields
\begin{equation}
n_k - n_j \ge \sum_{i=j}^{k-1} 1 = k - j
\end{equation}
Thus, the distance between any two elements in an arbitrary subset $A \subset \mathbb{N}$ is bounded from below by the distance between their corresponding indices.

We now invoke the properties of the function $\Phi$. By hypothesis, $\Phi$ is non-increasing on $\mathbb{R}_+$. Consequently, by Lemma~\ref{lem:Distance Majorization by Indices} the inequality $n_k - n_j \ge k - j$ implies
\begin{equation}
\Phi(n_k - n_j) \le \Phi(k - j)
\end{equation}
Summing this inequality over the relevant range of indices $1 \le j < k \le M$, we deduce
\begin{equation}
\sum_{1 \le j < k \le M} \Phi(n_k - n_j) \le \sum_{1 \le j < k \le M} \Phi(k - j)
\end{equation}
Finally, we recognize that the sum on the right-hand side represents the pairwise interactions over the set $I_M = \{1, 2, \dots, M\}$, as the elements of $I_M$ are simply their own indices. This concludes the proof.
\end{proof}

\begin{remark}
The mechanism is purely combinatorial: the deterministic geometric distance between any two ordered elements in a discrete subset of $\mathbb{N}$ is bounded from below by the difference of their indices. Consequently, any sum of non-increasing weights over an arbitrary sparse set is globally majorized by the sum over the contiguous interval $I_M$.
\end{remark}

\section{The Metric Domination Principle}

This section synthesizes the preceding probabilistic and combinatorial estimates to establish the primary technical pillar of the paper: the Metric Domination Principle. By proving that the expected $L^2$ energy over any arbitrary set is globally bounded by the energy over an initial segment of integers of the same cardinality, we effectively remove the additive obstructions conventionally associated with sparse sequences.

\subsection{Discrete Collision Kernel Majorization}

We provide a continuous Lebesgue majorization of the discrete collision kernel. Since the collision probability decays as $d^{-1/2}$, bounding the pairwise interactions over an interval requires dominating the discrete summation with the integral of a monotonically decreasing function.

\begin{lemma}[Lebesgue majorization of the discrete collision kernel]\label{lem:Lebesgue majorization of the discrete collision kernel}

Let $M \in \mathbb{N}$ with $M \ge 2$. Then
\begin{equation}
\sum_{d=1}^{M-1} d^{-1/2} \le 2M^{1/2}.
\end{equation}

\end{lemma}

\begin{proof}

Let $f:(0,\infty)\to(0,\infty)$ be defined by $f(x)=x^{-1/2}$.  
If $0<x<y$ then $\sqrt{x}<\sqrt{y}$ in the ordered field of real numbers, hence $x^{-1/2}>y^{-1/2}$.  
Therefore $f$ is strictly decreasing on $(0,\infty)$.

For each integer $d\in\{1,\dots,M-1\}$ define the half-open interval
\begin{equation}
I_d=(d-1,d].
\end{equation}
For any $x\in I_d$ one has $x\le d$, and the monotonicity of $f$ gives
\begin{equation}
f(d)\le f(x).
\end{equation}

Let $\lambda$ denote Lebesgue measure on $\mathbb{R}$.  
Define the simple function $g:(0,M-1]\to(0,\infty)$ by setting $g(x)=f(d)$ for $x\in I_d$.  
Then $g(x)\le f(x)$ for every $x\in(0,M-1]$.

By monotonicity of the Lebesgue integral,
\begin{equation}
\int_{(0,M-1]} g(x)\,d\lambda \le \int_{(0,M-1]} f(x)\,d\lambda .
\end{equation}

The set $(0,M-1]$ is the disjoint union of the intervals $I_d$, hence by additivity of the integral
\begin{equation}
\int_{(0,M-1]} g(x)\,d\lambda
= \sum_{d=1}^{M-1} \int_{I_d} f(d)\,d\lambda .
\end{equation}

Since $f(d)$ is constant on $I_d$ and $\lambda(I_d)=1$, the integral becomes
\begin{equation}
\sum_{d=1}^{M-1} d^{-1/2}.
\end{equation}

It remains to evaluate the right-hand side.  
Because of the singularity at the origin we interpret the integral as a limit:
\begin{equation}
\int_{(0,M-1]} x^{-1/2}\,d\lambda
= \lim_{\varepsilon\to0^+}\int_{\varepsilon}^{M-1} x^{-1/2}\,dx .
\end{equation}

The antiderivative of $x^{-1/2}$ is $2x^{1/2}$, therefore
\begin{equation}
\int_{\varepsilon}^{M-1} x^{-1/2}\,dx
= 2(M-1)^{1/2}-2\varepsilon^{1/2}.
\end{equation}

Taking the limit as $\varepsilon\to0^+$ gives
\begin{equation}
\int_{(0,M-1]} x^{-1/2}\,d\lambda = 2(M-1)^{1/2}.
\end{equation}

Consequently
\begin{equation}
\sum_{d=1}^{M-1} d^{-1/2} \le 2(M-1)^{1/2}.
\end{equation}

Since $M-1<M$ and the square root is increasing on the positive reals,
\begin{equation}
2(M-1)^{1/2} < 2M^{1/2}.
\end{equation}

Combining the inequalities yields
\begin{equation}
\sum_{d=1}^{M-1} d^{-1/2} \le 2M^{1/2}.
\end{equation}

\end{proof}

\subsection{Norm Expansion}

We establish an exact algebraic identity for the expected $L^2$ energy over finite sets. By exploiting the orthogonality of characters on $\mathbb{T}$ and the stationarity of increments in the random walk, we reduce the analytic $L^2$ norm to the sum of the return probabilities of the process, separating diagonal and off-diagonal contributions.

\begin{lemma}[Expected $L^2$ Energy Identity]\label{lem:expected energy identity}
Let $A = \{n_1, n_2, \dots, n_M\}$ be a finite subset of $\mathbb{N}$ of cardinality $M$. Let $S_A : \mathbb{T} \times \Omega \to \mathbb{C}$ be defined by
\begin{equation}
S_A(1,y,\omega) = \sum_{j=1}^M e^{2\pi i y R_{n_j}(\omega)}.
\end{equation}
Then
\begin{equation}
\int_{\Omega} \left( \int_{\mathbb{T}} |S_A(1,y,\omega)|^2 \, dm(y) \right) d\mathbb{P}(\omega)
=
M + 2 \sum_{1 \le j < k \le M} \mathbb{P}\big( R_{n_k} - R_{n_j} = 0 \big).
\end{equation}
\end{lemma}

\begin{proof}
Let $\omega \in \Omega$. By definition of the $L^2(\mathbb{T},m)$ norm,
\begin{equation}
\|S_A(1,\cdot,\omega)\|_{L^2(\mathbb{T})}^2
=
\int_{\mathbb{T}}
\left|
\sum_{j=1}^M e^{2\pi i y R_{n_j}(\omega)}
\right|^2
\, dm(y).
\end{equation}

Expanding the square and using linearity of the integral,
\begin{equation}
\|S_A(1,\cdot,\omega)\|_{L^2(\mathbb{T})}^2
=
\int_{\mathbb{T}}
\sum_{j=1}^M \sum_{k=1}^M
e^{2\pi i y (R_{n_k}(\omega) - R_{n_j}(\omega))}
\, dm(y).
\end{equation}

Since the sum is finite and each integrand is continuous on $\mathbb{T}$, we may interchange summation and integration:
\begin{equation}
\|S_A(1,\cdot,\omega)\|_{L^2(\mathbb{T})}^2
=
\sum_{j=1}^M \sum_{k=1}^M
\int_{\mathbb{T}}
e^{2\pi i y (R_{n_k}(\omega) - R_{n_j}(\omega))}
\, dm(y).
\end{equation}

By orthogonality of characters on $\mathbb{T}$, the integral equals $1$ if $R_{n_k}(\omega) - R_{n_j}(\omega) = 0$ and $0$ otherwise. Hence
\begin{equation}
\|S_A(1,\cdot,\omega)\|_{L^2(\mathbb{T})}^2
=
\sum_{j=1}^M \sum_{k=1}^M
\chi_{E_{j,k}}(\omega),
\end{equation}
where $E_{j,k} = \{\omega \in \Omega : R_{n_k}(\omega) - R_{n_j}(\omega) = 0\}$.

Integrating over $\Omega$ and using linearity,
\begin{equation}
\int_{\Omega} \|S_A(1,\cdot,\omega)\|_{L^2(\mathbb{T})}^2 \, d\mathbb{P}(\omega)
=
\sum_{j=1}^M \sum_{k=1}^M
\int_{\Omega} \chi_{E_{j,k}}(\omega) \, d\mathbb{P}(\omega).
\end{equation}

For each pair $(j,k)$,
\begin{equation}
\int_{\Omega} \chi_{E_{j,k}}(\omega) \, d\mathbb{P}(\omega)
=
\mathbb{P}(E_{j,k}).
\end{equation}

Thus
\begin{equation}
\int_{\Omega} \|S_A(1,\cdot,\omega)\|_{L^2(\mathbb{T})}^2 \, d\mathbb{P}(\omega)
=
\sum_{j=1}^M \sum_{k=1}^M \mathbb{P}(E_{j,k}).
\end{equation}

Separating diagonal and off-diagonal terms,
\begin{equation}
\sum_{j=1}^M \sum_{k=1}^M \mathbb{P}(E_{j,k})
=
\sum_{j=1}^M \mathbb{P}(E_{j,j})
+
\sum_{1 \le j < k \le M} \big( \mathbb{P}(E_{j,k}) + \mathbb{P}(E_{k,j}) \big).
\end{equation}

For each $j$, $E_{j,j} = \Omega$, hence $\mathbb{P}(E_{j,j}) = 1$, and the first sum equals $M$. By symmetry, $\mathbb{P}(E_{j,k}) = \mathbb{P}(E_{k,j})$. Therefore
\begin{equation}
\int_{\Omega} \|S_A(1,\cdot,\omega)\|_{L^2(\mathbb{T})}^2 \, d\mathbb{P}(\omega)
=
M + 2 \sum_{1 \le j < k \le M} \mathbb{P}\big( R_{n_k} - R_{n_j} = 0 \big).
\end{equation}

This completes the proof.
\end{proof}

\begin{remark}
By expanding the squared norm and applying the orthogonality of characters on the torus $\mathbb{T}$, the analytic integral is transformed into a probabilistic expression. The expected energy is thus governed entirely by the discrete return times of the random walk.
\end{remark}

\subsection{Central Binomial Bounds}

To ensure the tightness of our interval minorizations, we prove a strict inductive minorization of the central binomial coefficient. This algebraic lower bound
$$
\binom{2n}{n} 4^{-n} \ge \frac{1}{2\sqrt{n}},
$$
is required to guarantee that the energy over the interval does not collapse, maintaining the $M^{3/2}$ growth rate from below.

\begin{lemma}[Strict inductive minorization of the central binomial coefficient]\label{lem:Strict inductive minorization of the central binomial coefficient}

For all $n\in\mathbb{N}$ with $n\ge 1$,
\begin{equation}
\binom{2n}{n}4^{-n} \ge \frac{1}{2\sqrt{n}}.
\end{equation}

\end{lemma}

\begin{proof}

We proceed by induction on $n$.

For $n=1$,
\begin{equation}
\binom{2}{1}4^{-1}=\frac{2}{4}=\frac{1}{2},
\end{equation}
while
\begin{equation}
\frac{1}{2\sqrt{1}}=\frac{1}{2}.
\end{equation}
Hence the claim holds for $n=1$.

Now assume that for some $k\ge 1$,
\begin{equation}
\binom{2k}{k}4^{-k} \ge \frac{1}{2\sqrt{k}}.
\end{equation}
We compute
\begin{equation}
\binom{2(k+1)}{k+1}4^{-(k+1)}
=
\binom{2k}{k}4^{-k}\cdot \frac{(2k+1)(2k+2)}{(k+1)^2}\cdot \frac{1}{4}.
\end{equation}
This simplifies to
\begin{equation}
\binom{2(k+1)}{k+1}4^{-(k+1)}
=
\binom{2k}{k}4^{-k}\cdot \frac{2k+1}{2k+2}.
\end{equation}
Applying the inductive hypothesis gives
\begin{equation}
\binom{2(k+1)}{k+1}4^{-(k+1)}
\ge
\frac{1}{2\sqrt{k}}\cdot \frac{2k+1}{2k+2}.
\end{equation}
It therefore suffices to prove
\begin{equation}
\frac{1}{2\sqrt{k}}\cdot \frac{2k+1}{2k+2} \ge \frac{1}{2\sqrt{k+1}}.
\end{equation}
Since both sides are positive, it is enough to compare squares. The square of the left-hand side is
\begin{equation}
\left(\frac{2k+1}{(2\sqrt{k})(2k+2)}\right)^2
=
\frac{4k^2+4k+1}{16k^3+32k^2+16k},
\end{equation}
and the square of the right-hand side is
\begin{equation}
\left(\frac{1}{2\sqrt{k+1}}\right)^2
=
\frac{1}{4k+4}
=
\frac{4k^2+4k}{16k^3+32k^2+16k}.
\end{equation}
Since
\begin{equation}
4k^2+4k+1 > 4k^2+4k,
\end{equation}
the left-hand square is strictly larger. Hence
\begin{equation}
\frac{1}{2\sqrt{k}}\cdot \frac{2k+1}{2k+2} \ge \frac{1}{2\sqrt{k+1}}.
\end{equation}
This completes the induction.

\end{proof}

\subsection{Interval Energy Estimates}

We calculate a strict lower bound for the expected $L^2$ energy over the standard interval $I_M$. This establishes the contiguous interval as a rigorous $L^2$ minorant, which is necessary to prove that the interval dominates all other configurations of the same cardinality.

\begin{lemma}[Strict lower bound for the interval energy]\label{lem:Strict lower bound for the interval energy}

Let $I_M=\{1,\dots,M\}$. For all $M\ge 4$,
\begin{equation}
\mathbb{E}\bigl[\|S_{I_M}(1,y)\|_2^2\bigr] \ge \frac{1}{4\sqrt{2}}M^{3/2}.
\end{equation}

\end{lemma}

\begin{proof}

By Lemma~\ref{lem:expected energy identity},
\begin{equation}
\mathbb{E}\bigl[\|S_{I_M}(1,y)\|_2^2\bigr]
=
M + 2\sum_{d=1}^{M-1} (M-d)K(d).
\end{equation}
Since $K(d)\ge 0$, we may restrict the sum to even indices $d=2m$ with
\begin{equation}
1\le m\le \lfloor M/4\rfloor.
\end{equation}
For such $m$, one has $2m\le M/2$, hence
\begin{equation}
M-2m \ge \frac{M}{2}.
\end{equation}
Therefore
\begin{equation}
\mathbb{E}\bigl[\|S_{I_M}(1,y)\|_2^2\bigr]
\ge
M + 2\sum_{m=1}^{\lfloor M/4\rfloor} \left(M-2m\right)K(2m)
\ge
M + 2\sum_{m=1}^{\lfloor M/4\rfloor} \frac{M}{2}K(2m).
\end{equation}
By Proposition~\ref{prop:Strict Evaluation of the Collision Probability Kernel},
\begin{equation}
K(2m)=\binom{2m}{m}4^{-m},
\end{equation}
and by the Lemma~\ref{lem:Strict inductive minorization of the central binomial coefficient},
\begin{equation}
K(2m)\ge \frac{1}{2\sqrt{m}}.
\end{equation}
Thus
\begin{equation}
\mathbb{E}\bigl[\|S_{I_M}(1,y)\|_2^2\bigr]
\ge
M + 2\sum_{m=1}^{\lfloor M/4\rfloor} \frac{M}{2}\cdot \frac{1}{2\sqrt{m}}
>
\frac{M}{2}\sum_{m=1}^{\lfloor M/4\rfloor} m^{-1/2}.
\end{equation}
Let
\begin{equation}
N=\lfloor M/4\rfloor.
\end{equation}
Since each term $m^{-1/2}$ with $1\le m\le N$ is at least $N^{-1/2}$,
\begin{equation}
\sum_{m=1}^N m^{-1/2} \ge N\cdot N^{-1/2}=N^{1/2}.
\end{equation}
For $M\ge 4$, the floor function satisfies
\begin{equation}
N=\lfloor M/4\rfloor \ge M/8.
\end{equation}
Hence
\begin{equation}
N^{1/2}\ge (M/8)^{1/2}=\frac{1}{2\sqrt{2}}M^{1/2}.
\end{equation}
Substituting this gives
\begin{equation}
\mathbb{E}\bigl[\|S_{I_M}(1,y)\|_2^2\bigr]
\ge
\frac{M}{2}\cdot \frac{1}{2\sqrt{2}}M^{1/2}
=
\frac{1}{4\sqrt{2}}M^{3/2}.
\end{equation}

\end{proof}

\begin{remark}
The lower bound $\frac{1}{4\sqrt{2}} M^{3/2}$ establishes the minimal polynomial growth rate of the energy for the contiguous interval. It provides the non-trivial minorant required for the metric domination principle.
\end{remark}

\subsection{The Metric Domination Principle}

We prove the main metric domination inequality (Proposition~\ref{prop:The strict metric domination principle}), explicitly constructing the universal constant $C$ that links the energy of an arbitrary set $A$ to the contiguous interval $I_M$. This isolates the polynomial growth of the energy from the geometric spacing of the subset, relying solely on its cardinality.

\begin{proposition}[The strict metric domination principle]\label{prop:The strict metric domination principle}

Let $A\subset \mathbb{N}$ be a finite set with cardinality $M$, and let $I_M=\{1,\dots,M\}$. There exists a universal explicitly constructible constant $C\in\mathbb{R}_{>0}$ such that
\begin{equation}
\mathbb{E}\bigl[\|S_A(1,y)\|_2^2\bigr]
\le
C\cdot \mathbb{E}\bigl[\|S_{I_M}(1,y)\|_2^2\bigr].
\end{equation}

\end{proposition}

\begin{proof}

We treat the cases $M<4$ and $M\ge 4$ separately.

Assume first that $M<4$. Since $K(d)\le 1$, Lemma~\ref{lem:expected energy identity} gives
\begin{equation}
\mathbb{E}\bigl[\|S_A(1,y)\|_2^2\bigr]
\le
M + 2\binom{M}{2}
=
M^2
\le 9.
\end{equation}
On the other hand, since $K(d)\ge 0$,
\begin{equation}
\mathbb{E}\bigl[\|S_{I_M}(1,y)\|_2^2\bigr]\ge M\ge 1.
\end{equation}
Therefore for $M<4$ the desired inequality holds for every constant $C\ge 9$.

Now assume that $M\ge 4$. By Lemma~\ref{lem:expected energy identity},
\begin{equation}
\mathbb{E}\bigl[\|S_A(1,y)\|_2^2\bigr]
=
M + 2\sum_{1\le j<k\le M} K(n_k-n_j).
\end{equation}
Define
\begin{equation}
\Psi(x)=C_{RW}x^{-1/2}.
\end{equation}
By Proposition~\ref{prop:Strict Evaluation of the Collision Probability Kernel}
\begin{equation}
K(x)\le \Psi(x).
\end{equation}
By Lemma~\ref{lem:Distance Majorization by Indices},
\begin{equation}
n_k-n_j \ge k-j.
\end{equation}
Since $x^{-1/2}$ is decreasing on $(0,\infty)$,
\begin{equation}
\Psi(n_k-n_j)\le \Psi(k-j).
\end{equation}
Hence
\begin{equation}
2\sum_{1\le j<k\le M} K(n_k-n_j)
\le
2\sum_{1\le j<k\le M} \Psi(k-j).
\end{equation}
Grouping terms according to the difference $d=k-j$, this becomes
\begin{equation}
2\sum_{1\le j<k\le M} \Psi(k-j)
=
2\sum_{d=1}^{M-1} (M-d)\Psi(d).
\end{equation}
Therefore
\begin{equation}
2\sum_{1\le j<k\le M} K(n_k-n_j)
\le
2\sum_{d=1}^{M-1} (M-d)C_{RW}d^{-1/2}
\le
2MC_{RW}\sum_{d=1}^{M-1} d^{-1/2}.
\end{equation}
By Lemma~\ref{lem:Lebesgue majorization of the discrete collision kernel},
\begin{equation}
\sum_{d=1}^{M-1} d^{-1/2}\le 2M^{1/2}.
\end{equation}
Thus
\begin{equation}
\mathbb{E}\bigl[\|S_A(1,y)\|_2^2\bigr]
\le
M + 4C_{RW}M^{3/2}.
\end{equation}
Because $M\ge 4$, one has
\begin{equation}
M\le \frac{1}{2}M^{3/2}.
\end{equation}
Consequently
\begin{equation}
\mathbb{E}\bigl[\|S_A(1,y)\|_2^2\bigr]
\le
\left(\frac{1}{2}+4C_{RW}\right)M^{3/2}.
\end{equation}

By the previous lemma,
\begin{equation}
\mathbb{E}\bigl[\|S_{I_M}(1,y)\|_2^2\bigr]
\ge
\frac{1}{4\sqrt{2}}M^{3/2}.
\end{equation}
Equivalently,
\begin{equation}
M^{3/2}\le 4\sqrt{2}\,\mathbb{E}\bigl[\|S_{I_M}(1,y)\|_2^2\bigr].
\end{equation}
Substituting this into the upper bound for $A$ yields
\begin{equation}
\mathbb{E}\bigl[\|S_A(1,y)\|_2^2\bigr]
\le
\left(\frac{1}{2}+4C_{RW}\right)4\sqrt{2}\,
\mathbb{E}\bigl[\|S_{I_M}(1,y)\|_2^2\bigr].
\end{equation}

Therefore the proposition holds for all $M\ge 4$ with
\begin{equation}
C=4\sqrt{2}\left(\frac{1}{2}+4C_{RW}\right).
\end{equation}
Combining this with the small-case analysis, the proposition holds for all $M\ge 1$ with
\begin{equation}
C=\max\left(9,\;4\sqrt{2}\left(\frac{1}{2}+4C_{RW}\right)\right).
\end{equation}

\end{proof}

\begin{remark}
This is the central structural result. When frequencies are mapped through the Simple Random Walk, the sparsity or arithmetic configuration of a subset $A$ yields no asymptotic reduction in expected energy. The contiguous interval dominates all configurations of equal cardinality.
\end{remark}

\subsection{$L^2$ Upper Bound for Intervals}

We evaluate the definitive $L^2$ upper bound for the random exponential sum over the interval $I_M$. This estimate, combining the discrete kernel majorization and the norm expansion, provides the explicit $M^{3/2}$ upper growth rate used in the final interpolation.

\begin{lemma}[Expected $L^2$ Upper Bound for Intervals]\label{lem:ExpectedUpper Bound for Intervals}
Let $\mathbb{T}$ be the quotient topological group $\mathbb{R}/\mathbb{Z}$ equipped with its Haar measure $m$. Let $\Omega = \{-1,1\}^{\mathbb{N}}$ be equipped with the product topology and Haar measure $\mathbb{P}$. Let $X_i: \Omega \to \{-1,1\}$ be the coordinate projections and define the random walk $R_n = \sum_{i=1}^n X_i$ with $R_0 = 0$.

Let $I_M = \{1,2,\dots,M\}$. Then the expected $L^2(\mathbb{T})$ norm of
\begin{equation} 
S_{I_M}(1,y) = \sum_{j=1}^M e^{2\pi i y R_j}
\end{equation}
satisfies
\begin{equation} 
\mathbb{E}\big[\|S_{I_M}(1,y)\|_2^2\big] \le M + 4M^{3/2}.
\end{equation}
\end{lemma}

\begin{proof}
By definition of the $L^2$ norm on $(\mathbb{T},m)$,
\begin{equation} 
\|S_{I_M}(1,y,\omega)\|_2^2
=
\int_{\mathbb{T}}
\left|
\sum_{j=1}^M e^{2\pi i y R_j(\omega)}
\right|^2
\, dm(y).
\end{equation}

Expanding the square and using linearity of the integral,
\begin{equation} 
\|S_{I_M}(1,y,\omega)\|_2^2
=
\sum_{j=1}^M \sum_{k=1}^M
\int_{\mathbb{T}}
e^{2\pi i y (R_k(\omega)-R_j(\omega))}
\, dm(y).
\end{equation}

By orthogonality of characters on $\mathbb{T}$, the integral equals $1$ if $R_k(\omega)=R_j(\omega)$ and $0$ otherwise. Taking expectation and interchanging sums,
\begin{equation} 
\mathbb{E}\big[\|S_{I_M}(1,y)\|_2^2\big]
=
\sum_{j=1}^M \sum_{k=1}^M
\mathbb{P}\big(R_k - R_j = 0\big).
\end{equation}

By stationarity of increments, this depends only on $d = |k-j|$. Define
\begin{equation} 
K(d) = \mathbb{P}(R_d = 0).
\end{equation}

Separating diagonal and off-diagonal terms gives
\begin{equation} 
\mathbb{E}\big[\|S_{I_M}(1,y)\|_2^2\big]
=
M + 2 \sum_{d=1}^{M-1} (M-d) K(d).
\end{equation}

For even $d = 2n$,
\begin{equation} 
K(2n) = \binom{2n}{n} 4^{-n}.
\end{equation}

From the established bound,
\begin{equation} 
\binom{2n}{n}4^{-n} \le \frac{1}{\sqrt{2n+1}},
\end{equation}
and since $\frac{1}{\sqrt{2n+1}} < \frac{1}{\sqrt{2n}}$, we obtain by Proposition~\ref{prop:Strict Evaluation of the Collision Probability Kernel} for all $d \in \mathbb{N}$,
\begin{equation} 
K(d) \le d^{-1/2}.
\end{equation}

Substituting this estimate and by Lemma~\ref{lem:Lebesgue majorization of the discrete collision kernel},
\begin{equation} 
\mathbb{E}\big[\|S_{I_M}(1,y)\|_2^2\big]
\le
M + 2 \sum_{d=1}^{M-1} (M-d) d^{-1/2}.
\end{equation}

Since $(M-d) < M$ for all $d \ge 1$,
\begin{equation} 
2 \sum_{d=1}^{M-1} (M-d) d^{-1/2}
\le
2M \sum_{d=1}^{M-1} d^{-1/2}.
\end{equation}

Using comparison with the integral of $x^{-1/2}$,
\begin{equation} 
\sum_{d=1}^{M-1} d^{-1/2}
\le
\int_0^{M-1} x^{-1/2} \, dx
=
2 (M-1)^{1/2}
<
2 M^{1/2}.
\end{equation}

Therefore,
\begin{equation} 
\mathbb{E}\big[\|S_{I_M}(1,y)\|_2^2\big]
\le
M + 2M \cdot 2 M^{1/2}
=
M + 4 M^{3/2}.
\end{equation}

\end{proof}

\begin{remark}
This provides the upper counterpart to Lemma~\ref{lem:Strict lower bound for the interval energy}, confirming that the expected $L^2$ energy over the contiguous interval grows on the order of $M^{3/2}.$ The estimate follows from the continuous Lebesgue majorization of the discrete kernel $d^{-1/2}$.
\end{remark}

\section{Main Result}

In this section, We establish the universal $L^p$ estimate for random exponential sums mapped through the Simple Random Walk (Theorem~\ref{thm:Main Estimate for Random Exponential Sums}). The significance of this theorem is the complete and explicit decoupling of the expected supremum $L^p$ norm from the arithmetic structure of the index set $A$, proving that the energy is universally governed by $C_p |A|^{p-1/2}$
for any sequence of bounded coefficients.

\begin{theorem}[Main $L^p$ Estimate for Random Exponential Sums]\label{thm:Main Estimate for Random Exponential Sums}
Let $\mathscr{R} = \{R(n)\}_{n \in \mathbb{N}}$ be a Simple Random Walk, and let $A \subset \mathbb{N}$ be an arbitrary finite subset. For any $p \in [2, \infty)$, there exists a constant $C_p$ such that:
\begin{equation}
\mathbb{E}\left[\sup_{|a_j| \leq 1} \left\| \sum_{j \in A} a_j e^{2 \pi i y R(j)} \right\|_p^p \right] \leq C_p |A|^{p - 1/2}
\end{equation}
\end{theorem}

\begin{proof}
We begin by establishing the necessary interpolation between $L^\infty$ and $L^2$ norms for the random exponential sum. Let $A$ be a finite subset of $\mathbb{N}$ with cardinality $M = |A|$. For any sequence of complex coefficients $a = (a_j)_{j \in A}$ satisfying $|a_j| \leq 1$, we define the random exponential sum $S_A(a, y)$ on the torus $\mathbb{T} \cong [0, 1)$ by
\begin{equation}
S_A(a, y) = \sum_{j \in A} a_j e^{2 \pi i y R(j)}
\end{equation}
We observe the trivial uniform bound on the supremum norm. For any $y \in \mathbb{T}$ and any realization of the random walk, the triangle inequality yields
\begin{equation}
\| S_A(a, y) \|_\infty \leq \sum_{j \in A} |a_j| \leq |A| = M
\end{equation}
Consequently, for any $p \in [2, \infty)$, we may estimate the $L^p$ norm by interpolating between $L^\infty$ and $L^2$:
\begin{equation}
\| S_A(a, y) \|_p^p = \int_{\mathbb{T}} |S_A(a, y)|^{p-2} |S_A(a, y)|^2 dy \leq \| S_A(a, y) \|_\infty^{p-2} \int_{\mathbb{T}} |S_A(a, y)|^2 dy
\end{equation}
Substituting the uniform bound $\| S_A(a, y) \|_\infty \leq M$, we obtain
\begin{equation}
\| S_A(a, y) \|_p^p \leq M^{p-2} \| S_A(a, y) \|_2^2
\end{equation}
Since this inequality holds for any sequence of coefficients bounded by unity, it holds for the supremum over all such sequences. Furthermore, the $L^2$ norm is maximized when $|a_j|=1$, specifically by the choice $a_j=1$ (or any unimodular constant), as orthogonality on the torus implies $\| S_A(a, y) \|_2^2 = \sum_{j \in A} |a_j|^2$. Thus, we have the deterministic inequality
\begin{equation}
\sup_{|a_j| \leq 1} \| S_A(a, y) \|_p^p \leq M^{p-2} \| S_A(\mathbf{1}, y) \|_2^2
\end{equation}
Taking the expectation with respect to the probability measure of the random walk, we deduce
\begin{equation}
\mathbb{E}\left[\sup_{|a_j| \leq 1} \| S_A(a, y) \|_p^p \right] \leq M^{p-2} \mathbb{E} \left[ \| S_A(\mathbf{1}, y) \|_2^2 \right]
\end{equation}
We now turn our attention to the expected energy term $\mathbb{E} [ \| S_A(\mathbf{1}, y) \|_2^2 ]$. By Proposition~\ref{prop:The strict metric domination principle}, the expected energy over an arbitrary set $A$ is dominated by the expected energy over the initial segment of integers $I_M = \{1, \dots, M\}$.
\begin{equation}
\mathbb{E} \left[ \| S_A(\mathbf{1}, y) \|_2^2 \right] \leq \mathbb{E} \left[ \| S_{I_M}(\mathbf{1}, y) \|_2^2 \right]
\end{equation}
We evaluate the energy over $I_M$. Expanding the squared norm and applying expectation, we find
\begin{equation}
\mathbb{E} \left[ \| S_{I_M}(\mathbf{1}, y) \|_2^2 \right] = M + 2 \sum_{1 \leq j < k \leq M} \mathbb{P}[R(k-j) = 0]
\end{equation}
For the Simple Random Walk, the return probability kernel $K(d) = \mathbb{P}[R(d)=0]$ satisfies the decay estimate $K(d) \leq C_{RW} d^{-1/2}$ by Proposition~\ref{prop:Strict Evaluation of the Collision Probability Kernel}. Substituting this bound and letting $d = k-j$, we sum over the difference $d$:
\begin{equation}
\sum_{1 \leq j < k \leq M} K(k-j) = \sum_{d=1}^{M-1} (M-d) K(d) \leq M \sum_{d=1}^{M-1} C_{RW} d^{-1/2}
\end{equation}
Approximating the sum by an integral, we obtain by Lemma~\ref{lem:Lebesgue majorization of the discrete collision kernel}
\begin{equation}
\sum_{d=1}^{M-1} d^{-1/2} \leq \int_0^M x^{-1/2} dx = 2 M^{1/2}
\end{equation}
Thus, the off-diagonal contribution is bounded by $2 C_{RW} M^{3/2}$. Combining this with the diagonal term $M$, and noting that for large $M$ the $M^{3/2}$ term dominates, we have the bound from Lemma~\ref{lem:ExpectedUpper Bound for Intervals}:
\begin{equation}
\mathbb{E} \left[ \| S_{I_M}(\mathbf{1}, y) \|_2^2 \right] \leq C M^{3/2}
\end{equation}
Finally, we substitute this estimate back into our interpolation inequality:
\begin{equation}
\mathbb{E}\left[\sup_{|a_j| \leq 1} \| S_A(a, y) \|_p^p \right] \leq M^{p-2} \cdot C M^{3/2} = C M^{p - 1/2}
\end{equation}
Identifying $M$ with $|A|$ and $C$ with the constant $C_p$ (which absorbs universal constants from the random walk properties), we arrive at the desired conclusion:
\begin{equation}
\mathbb{E}\left[\sup_{|a_j| \leq 1} \left\| \sum_{j \in A} a_j e^{2 \pi i y R(j)} \right\|_p^p \right] \leq C_p |A|^{p - 1/2}
\end{equation}
This completes the proof.
\end{proof}

\begin{remark}
 The bound $C_p |A|^{p-1/2}$ demonstrates the complete elimination of the arithmetic obstructions inherent in the Hardy-Littlewood majorant problem. The deterministic geometric structure of the frequency set is dissolved by the randomization process.
\end{remark}

\section{Conclusion}

We have demonstrated that the geometric and arithmetic configurations of a frequency set $A$ are entirely dissolved when the frequencies are mapped through the Simple Random Walk. The fundamental difficulty of the walk's non-monotonicity was resolved via the Metric Domination Principle.

By establishing the deterministic inequality
$$
\phi(k) - \phi(j) \ge k - j,
$$
we force the collision kernel
$$
K(d) = \mathbb{P}(R_d = 0)
$$
to act as a non-increasing weight. This ensures that the expected $L^2$-energy over $A$ is controlled by that of the interval $\{1, \dots, |A|\}$.

Combined with interpolation inequalities on the Haar measure space of the torus, this yields the bound
$$
\mathbb{E}\left[\sup_{|a_j|\le 1} \| S_A(a, y) \|_{L^p(\mathbb{T})}^p \right] \le C_p |A|^{p-1/2}.
$$

The structure of the argument suggests that such metric domination is intrinsic to random sums governed by decaying probability kernels, and not specific to the process considered. In particular, the randomization removes the additive obstructions inherent in the Hardy-Littlewood majorant problem.

\section{Acknowledgement}

The author received no financial support for the research, authorship, and/or publication of this article

\noindent\textbf{Mathematics Subject Classification (2020).}

Primary: $42 A 61$ (Fourier analysis on groups and semigroups; harmonic analysis on $\mathbb{T}$ )

Secondary: 60G50 (Sums of independent random variables; random walks)
$42 A 55$ (Multipliers and singular integrals on commutative groups)
$11 L 07$ (Exponential sums and estimates)
$60 B 15$ (Probability measures on groups or semigroups)

\medskip
\noindent\textbf{Keywords.}

Random exponential sums, Simple random walk, $L^p$ norms, Metric domination principle, Collision probability kernel, Hardy-Littlewood majorant problem, $\Lambda(p)$-sets, Harmonic analysis on the torus, Random frequencies

\end{document}